\documentclass[a4paper,11pt,reqno]{amsart}
\usepackage{amsmath}%
\usepackage{amsfonts}%
\usepackage{amssymb}%
\usepackage{graphicx}
\usepackage{hyperref}
\usepackage[margin=1.25in]{geometry}
\theoremstyle{remark}
\newtheorem{remark}{Remark}

\theoremstyle{definition}
\newtheorem{definition}{Definition}
\newtheorem{assumption}{Assumption}

\theoremstyle{plain}
\newtheorem{theorem}{Theorem}

\newtheorem{corollary}{Corollary}

\newtheorem{openproblem}{Open Problem}

\numberwithin{equation}{section}

\newcommand{\bint}[1]{\oint_{\partial\Omega}#1\;\text{dS}\,}
\newcommand{\vint}[1]{\int_{\Omega}#1\;\text{d}x\,}
\newcommand{\abs}[1]{\left|#1\right|}
\newcommand{\norm}[1]{\left\lVert#1\right\rVert}

\newcommand{\bintk}[1]{\oint_{\partial\Omega_k}#1\;\text{dS}\,}

\newcommand{\vintk}[1]{\int_{\Omega_k}#1\;\text{d}x\,}

\title{Gradient Flows for Dirichlet Eigenvalues}

\author{Yannick Holle}
\address[Y. Holle]{\newline
\indent RWTH Aachen University \newline%
\indent Institut f{\"u}r Mathematik\newline
\indent Templergraben 55\newline
\indent D-52062 Aachen}%
\email{holle@eddy.rwth-aachen.de}%

\thanks{This work is part of the author’s Master's Thesis at the RWTH Aachen University. The author is very grateful to his advisor Alfred Wagner for many helpful discussions and his continuous support.
This work has been partially funded by the Deutsche Forschungsgemeinschaft
(DFG, German Research Foundation) Projektnummer 320021702/GRK2326
Energy, Entropy, and Dissipative Dynamics (EDDy).}
\date{\today}

\begin{document}

\maketitle
\thispagestyle{empty}
\begin{abstract}
We are interested in existence of gradient flows for shape functionals especially for first Laplacian eigenvalues. We introduce different techniques to prove existence and use different formulations for gradient flows. We apply a compactness argument to prove existence of generalized minimizing movements for Dirichlet and Robin boundary conditions with respect to several common metrics. Moreover, we use Brunn-Minkowski inequalities to prove $\alpha$-convexity and existence of contraction semi-groups for Dirichlet boundary conditions on convex bodies. Finally, we give a proof of $\alpha$-convexity for Robin boundary conditions by the second domain variation since we do not have a Brunn-Minkowski inequality in this case.
\end{abstract}

\section{Introduction}
\label{sec:Introduction}
The problem of finding the drum with the lowest fundamental frequency can be traced back to Lord Rayleigh \cite{Ra1894} and can be reformulated in finding the domain of given volume with the lowest Dirichlet Laplacian eigenvalue. This problem was solved independently by Krahn \cite{Kr1925} and Faber \cite{Fa1923} by the so-called Faber-Krahn inequality. Since this is a (shape) optimization Problem, it is a natural aim to study the corresponding energy landscape and its properties. We restrict ourselves to the question:
\textit{Can the optimal domain be found by deforming a given initial domain such that the infinitesimal decent is as steep as possible?}
These curves of steepest decent are called gradient flows. So, we are interested in gradient flows for the first Laplacian eigenvalue with certain boundary conditions.\\
Gradient flows are usually studied in Hilbert spaces, where Riesz's representation theorem defines a gradient for sufficiently smooth functionals. It turns out, that gradient flows can also be defined in metric spaces. One of the standard reference for this theory is the seminal book by Ambrosio, Gigli and Savare \cite{Ambrosio2005}. We use two common notions for gradient flows in metric spaces. The first one is the so-called (generalized) minimizing movement defined by an implicit Euler scheme, which exists under weak assumptions on the functional and the metric space. The second one is the solution to an evolution variational inequality. This notion is stronger and unique solutions exist if, among other assumptions, an $\alpha$-convexity property is satisfied. The evolution variational inequality induces an $\alpha$-contraction semi-group and gives more regularity than the (generalized) minimizing movement.\\
Gradient flows for shape functionals and particularly for Laplacian eigenvalues were first studied by Bucur, Buttazzo and Stefanelli \cite{Bucur2012}. They proved some general results for evolutions of capacity measures w.r.t. the $\gamma$-convergence and for spectral optimization problems. They studied generalized minimizing movements for spectral functionals which are decreasing in the sense of set inclusion and for spectral functionals on convex sets. Furthermore, they proved and disproved some geometric properties of these gradient flows. In the Master thesis of the author \cite{Holle2018}, new techniques are derived to prove $\alpha$-convexity of Laplacian eigenvalues which implies existence of contraction semi-groups. These results are presented in this paper in a slightly modified way and some details are omitted.\\
A key problem in studying gradient flows for shape problems is the choice of the metric space. Since there are many metrics on the space of domains which behave very differently and induce different topologies, the behavior of the obtained gradient flows can be very different as well. Especially, if we consider contraction semi-groups, the geometry induced by the metric is crucial. The involved $\alpha$-convexity condition controls the behavior of the functional along paths and contains a geometric restriction on the metric space.\\
We give two proofs of the $\alpha$-convexity property. The first one uses the Brunn-Minkowski inequality for the first Dirichlet eigenvalue on convex bodies.\linebreak Brunn-Minkowski inequalities are convexity/concavity inequalities and hold true for several shape functionals and eigenvalues on convex domains. For more examples like Newton capacity, torsional rigidity and Monge-Amp\`ere eigenvalue, see \cite{Colesanti2005}. The obtained gradient flows are with respect to the $L^2$-metric for convex bodies. The second approach uses domain perturbations and the second domain variation of the first Robin eigenvalue on a class of star-shaped sets. A global estimate for the second domain variation leads to the $\alpha$-convexity for a metric induced by the $W^{1,2}$-norm of boundary parametrization.\\
The paper is organized as follows: First, we introduce gradient flows in metric spaces (Section \ref{sec:GradientFlows}) and recall basic properties of Laplacian eigenvalues (Section \ref{sec:RobinEigenvalues}). In Section \ref{sec:RobinEigenvalueAndGMM}, we prove existence of generalized minimizing movements for domains with uniform $\epsilon$-cone property and constrained volume. Next, we consider contraction semi-groups and $\alpha$-convexity. The approaches use the Brunn-Minkowski inequality on convex bodies (Section \ref{sec:ConvexbodiesandtheBrunn-Minkowskiinequality}) and the second domain variation on star-shaped domains (Section \ref{sec:AlphaConvexityOfTheRobinEigenvalue}). 

\section{Gradient flows}
\label{sec:GradientFlows} 
We recall two notions of gradient flows in metric spaces which coincide with the notion of a gradient flow if a Hilbertian structure is given. \\
Throughout this section $(X,d)$ is a complete metric space, $u_0$ is an initial condition and $\phi\colon X\to (-\infty,+\infty]$ is a proper functional. A functional is called \textit{proper} if its \textit{effective domain} $D(F)=\{ x\in X \,|\,F(x)<+\infty\}$ is non-empty.

\subsection{Generalized Minimizing Movements}
\label{sec:GeneralizedMinimizingMovements}
For every fixed $h>0$ the \textit{implicit Euler scheme} of time step $h$ and initial condition $u_0\in X$ consists in constructing a function $u_h(t)=w(\lfloor t/h\rfloor)$ (where $\lfloor\cdot\rfloor$ denotes the floor function) by
\begin{equation}
w(0)=u_0, \quad  w(n+1)\in \operatorname*{argmin}\limits_{y\in X}\Phi(h,w(n),y),
\label{eq:}
\end{equation}
where 
\begin{equation}
\Phi(h,x;y):=\phi(y)+\frac{1}{2h}d^2(x,y),
\label{eq:DefinitionPhi}
\end{equation}
for $h>0$ and $x,y\in X$. 
\begin{definition}
For a given functional $\Phi$ as defined in $(\ref{eq:DefinitionPhi})$ and an initial condition $u_0\in X$,
we say that a curve $u\colon[0,+\infty)\rightarrow X$ is a \textit{minimizing movement} for $\Phi$ starting at $u_0$ if
\begin{equation}
u_h(t)\to u(t) \quad \text{as }h\to 0,\text{ for all }t\in[0,+\infty). 
\label{eq:}
\end{equation}
We denote the collection of all minimizing movements for $\Phi$ starting at $u_0$ by \linebreak $MM(\Phi;u_0)$.
Analogously, we say that a curve $u\colon[0,\infty)\rightarrow X$ is a \textit{generalized minimizing movement} for $\Phi$ starting at $u_0$ if the latter convergence holds for a subsequence $h_n\to 0$ and we write $u\in GMM(\Phi;u_0)$.
\end{definition}
We state a general existence result for generalized minimizing movements (see Proposition 2.2.3 in \cite{Ambrosio2005}).
\begin{theorem}\label{thm:ExistenceGMM}
Let $(X,d)$ be a complete metric space. Fix $u_0\in D(\phi)$ and let $\phi\colon X\to (-\infty,+\infty]$ be a proper function such that
\begin{itemize}
	\item $\phi\colon (X,d)\to (-\infty,+\infty]$ is lower semi-continuous;
	\item there exists $h_*>0$ and $x_*\in X$ such that 
		\begin{equation*}
			\inf_{y\in X}\Phi(h_*,x_*;y)>-\infty;
		\end{equation*}
	\item for every sequence $\{x_n\}\subset X$ with 
	\begin{equation*}
	\sup_n \phi(x_n)<+\infty\quad\text{ and } \quad\sup_{n,m}d(x_n,x_m)<+\infty,
	\end{equation*}
	there exists a convergent subsequence.	
\end{itemize}
Then, there exists a curve $u\in AC^2_{loc}([0,+\infty);X)$ with $u(0+)=u_0$ and a sequence $h_n\to 0$ such that
\begin{equation*}
u_{h_n}(t)\to u(t) \quad \text{as }n\to +\infty.
\end{equation*}
In particular, the set $GMM(\Phi;u_0)$ is non-empty.
\end{theorem}
We recall that the curves in $AC^2_{loc}([0,+\infty);X)$ are the absolutely continuous curves with locally square-integrable metric derivative. \\
Notice that the curve $u$ in the previous theorem is a so-called \textit{curve of maximal slope} with respect to the strong upper gradient $\abs{\partial F}$.

\subsection{Contraction semi-groups}
\label{sec:ContractionSemigroups}
In this subsection we are interested in the stronger notion of contraction semi-groups. The semi-groups are induced by minimizing movements with given initial data, they are the unique solution of an evolution variational inequality and they satisfy a contraction property. \\
We start with some definitions.
\begin{definition}
Let $\phi\colon X\rightarrow (-\infty,+\infty]$ be proper and lower semi-continuous and let $\alpha\in\mathbb R$. A function $u\in C([0,\infty),X)\cap AC_{loc}((0,\infty),X)$, which satisfies 
\begin{equation*}
u(0)\in \overline{D(\phi)},u(t)\in D(\phi)\quad \text{for every }t>0
\end{equation*}
and solves
\begin{equation}
\frac{1}{2}\frac{d}{dt}d^2(u(t),z)+\frac{\alpha}{2}d^2(u(t),z)+\phi(u(t))\le \phi(z)\quad\text{a.e. in }(0,\infty)
\label{eq:EVI}
\end{equation}
for every $z\in D(\phi)$, is called a solution to the \textit{evolution variational inequality}. 
\end{definition}
We introduce two assumptions which are involved in the main existence result.
\begin{assumption}
\label{ass:AlphaConvexity}
There exists $\alpha\in\mathbb R$ such that for every $x,y_0,y_1\in D(\phi)$ there exists a curve $\gamma\colon [0,1]\rightarrow D(\phi)$ satisfying $\gamma(0)=y_0$ and $\gamma(1)=y_1$ with
\begin{align}
\frac{1}{2h}d^2(x,\gamma(t))+&\phi(\gamma(t))\le(1-t)\left[\frac{1}{2h}d^2(x,y_0)+\phi(y_0)\right]\nonumber\\
&+t\left[\frac{1}{2h}d^2(x,y_1)+\phi(y_1)\right]-\left(\frac{1}{h}+\alpha\right)\frac{1}{2} t(1-t)d^2(y_0,y_1)
\label{eq:AssumptionH1}
\end{align}
for every $t\in[0,1]$ and $1+\alpha h>0$.
\end{assumption}
\begin{remark}\label{rmk:alphaconvexityHilbert}
The previous inequality arises through the generalization of $\alpha$-convex functions $\hat\phi\colon H\to\mathbb R$ on a Hilbert space $H$. The function $\hat\phi$ is called $\alpha$-convex iff $x\mapsto \hat\phi(x)-\tfrac{\alpha}{2}\langle x,x\rangle$ is convex on H. Obviously, all results in this subsection are true for $\alpha$-convex functions on convex sets in Hilbert spaces and $\gamma(t):=(1-t)y_0+ty_1$.\\
Notice that the limit $h\to 0$ in (\ref{eq:AssumptionH1}) gives a restriction to the curve $\gamma$ and the metric space $(X,d)$. We are particularly interested in Banach spaces where the choice $\gamma(t)=(1-t)y_0+t y_1$ does not longer work.
\end{remark}
\begin{assumption}
\label{ass:BoundedFromBelow}
There exist $x_*\in D(\phi),r_*>0$ and $m_*\in \mathbb R$ such that $\phi(y)\ge m_*$ for every $y\in X$ satisfying $d(x_*,y)\le r_*$.
\end{assumption}
We state the general existence result for contraction semi-groups (see Theorem 4.0.4 in \cite{Ambrosio2005}).
\begin{theorem}
Let $(X,d)$ be a complete metric space and let $\phi\colon X\rightarrow (-\infty,+\infty]$ be a proper and lower semi-continuous function which satisfies Assumption \ref{ass:AlphaConvexity} and Assumption \ref{ass:BoundedFromBelow}. Let $\alpha\in\mathbb R$ be as in Assumption \ref{ass:AlphaConvexity}. Then, the following holds.
\begin{enumerate}
	\item {\normalfont Convergence and exponential formula:} for each $u_0\in\overline{D(\phi)}$ there exists a unique element $u=S[u_0]$ in $MM(\Phi;u_0)$ which therefore can be expressed through the exponential formula
	\begin{equation*}
	u(t)=S[u_0](t)=\lim_{n\to\infty}(J_{t/n})^n[u_0].
	\label{eq:}
	\end{equation*}
	$J_h x$ denotes the unique global minimizer element of $X\ni y\mapsto \Phi(h,x;y)$.
	\item {\normalfont Regularizing effect:} $u$ is a locally Lipschitz curve of maximal slope with $u(t)\in D(\abs{\partial\phi})\subset D(\phi)$ for $t>0$.
	\item {\normalfont Uniqueness and evolution variational inequalities:} $u$ is the unique solution to (\ref{eq:EVI}) among all the locally absolutely continuous curves such that $\lim_{t\downarrow 0}u(t)=u_0$.
	\item {\normalfont Contraction semigroup:} the map $t\mapsto S[u_0](t)$ is an $\alpha$-contracting semi-group i.e.
	\begin{equation*}
	d(S[u_0](t),S[v_0](t))\le e^{-\alpha t}d(u_0,v_0)\quad \text{for all } u_0,v_0\in \overline{D(\phi)}
	\end{equation*}
	\item {\normalfont Optimal a priori estimate:} if $u_0\in D(\phi)$ and $\alpha=0$, then
	\begin{equation*}
	d^2\left(S[u_0](t),(J_{t/n})^n[u_0]\right)\le \frac{t}{n}\left(\phi(u_0)-\phi_{t/n}(u_0)\right)\le\frac{t^2}{2n^2}\abs{\partial\phi}^2\!(u_0),
	\end{equation*}
	where $\phi_h(x):=\inf_{y\in X}\Phi(h,x;y)$.
\end{enumerate}
\label{thm:ExistenceSemigroups}
\end{theorem}

\section{Laplacian Eigenvalues}
\label{sec:RobinEigenvalues}
We recall some basic properties of the Laplacian eigenvalues with Dirichlet and Robin boundary conditions.
We fix an open, bounded set $\Omega\subset \mathbb R^n$ with Lipschitz boundary and $n\in\mathbb N_{\ge 2}$.
A number $\lambda_D\in \mathbb R$ is called an eigenvalue of the Laplace operator with Dirichlet boundary conditions if there exists a non-trivial weak solution $u\in W^{1,2}(\Omega)$ to the problem
\begin{align}
\begin{cases}
\Delta u+\lambda_D u&=0 \quad\text{in }\Omega\\
u&=0\quad\text{on }\partial\Omega.
\end{cases}
\end{align}
The weak formulation of this problem is given by
\begin{equation}
\vint{\nabla u \nabla \phi}=\lambda_D\vint{u\phi}\quad \text{for all }\phi\in W^{1,2}_0(\Omega).
\label{eq:weakformulationRobinEV}
\end{equation}

A number $\lambda_R\in \mathbb R$ is called an eigenvalue of the Laplace operator with Robin boundary conditions and parameter $\beta>0$ if there exists a non-trivial weak solution $u\in W^{1,2}(\Omega)$ to the problem
\begin{align}
\begin{cases}
\Delta u+\lambda_R u&=0\quad\text{in }\Omega\\
\partial_\nu u+\beta u&=0\quad\text{on }\partial\Omega,
\end{cases}
\end{align}
where $\partial_\nu$ denotes the outer normal derivative on $\partial\Omega$. The weak formulation of this problem is given by
\begin{equation}
\vint{\nabla u \nabla \phi}+\beta\bint{u\phi}=\lambda_R\vint{u\phi}\quad \text{for all }\phi\in W^{1,2}(\Omega).
\label{eq:weakformulationRobinEV}
\end{equation}
Sometimes we do not want to distinguish between the boundary conditions and denote the eigenvalue with omitted indices $\lambda$. Furthermore, we write $\lambda=\lambda(\Omega)$ if we want to emphasize the dependence on the domain $\Omega$. We recall some more properties of the eigenvalues.
It is well known that the eigenvalue problems admit a discrete, diverging spectrum
\begin{equation}
0< \lambda_1\le\lambda_2\le\lambda_3\le \cdots\le \lambda_k\le\cdots\longrightarrow +\infty,
\label{eq:discretespectrum}
\end{equation}
which are given by the variational problems
\begin{align}
\lambda_{D,k}&=\min_{\substack{V\in \mathcal S_k\\V\subset W^{1,2}_0(\Omega)}}\max_{v\in V,v\neq 0}\frac{\vint{|\nabla v|^2}}{\vint{v^2}},\\
\lambda_{R,k}&=\min_{V\in \mathcal S_k}\max_{v\in V,v\neq 0}\frac{\vint{|\nabla v|^2}+\beta\bint{v^2}}{\vint{v^2}}, \label{eq:variationalEVproblem}
\end{align}
where $\mathcal S_k$ denotes the family of $k$-dimensional subspaces of $W^{1,2}(\Omega)$. We are particularly interested in the first eigenvalue $\lambda_1$ which is simple if $\Omega$ is connected. In the following we write $\lambda:=\lambda_1$ and do not always emphasis that we mean the first eigenvalue.\\
The Dirichlet boundary conditions are the simpler boundary conditions and are much better understood. They can be seen as the limit $\beta\to +\infty$ of the Robin boundary conditions. We collect several differences of the boundary conditions.\smallskip\\
\begin{tabular}{l|c|c}
 & $\lambda_D$ & $\lambda_R$\\
\hline
increasing w.r.t. set inclusion & yes & no \\
\hline
homogeneity & $\lambda_D(r\Omega)=\lambda_D(\Omega)/r^2$ & no, but $\lambda_R(r\Omega)\le\lambda_R(\Omega)/r,\,r\ge 1$\\
\hline
log-concave eigenfunctions & yes & no
\end{tabular}
\smallskip

If we restrict the class of domains to those with a given measure, the shape optimization problem
\begin{equation}
\min\{\lambda(\Omega)\,|\,\partial\Omega\text{ is Lipschitz}, |\Omega|=M\}
\label{eq:}
\end{equation}
arises. Its minimum is attained by every Ball of measure $M>0$ since the Krahn-Faber inequality
\begin{equation}
\lambda(B)\le\lambda(\Omega) \quad\text{for every Ball } B \text{ and Lipschitz domain }\Omega \text{ with }|\Omega|=|B|,
\label{eq:}
\end{equation}
holds (see \cite{Kr1925,Fa1923} for Dirichlet boundary conditions and \cite{Daners2006} for Robin boundary conditions).
\begin{remark}\label{rmk:negativebeta}
The consideration of (generalized) minimizing movements for the first Robin eigenvalue with $\beta<0$ is problematic because a fixed Lipschitz domain $\Omega$ can be approximated by sets $\Omega_\varepsilon$ for $\varepsilon>0$ such that
\begin{equation*}
d_H(\Omega,\Omega_\varepsilon)<\varepsilon\quad\text{and}\quad \abs{\partial\Omega}\to +\infty\quad \text{as }\varepsilon\to 0,
\end{equation*}
where $d_H$ denotes the Hausdorff metric for open sets.
This implies that the implicit Euler scheme is not well-defined, i.e. 
\begin{equation*}
\inf_{\hat\Omega}\Phi(h,\Omega;\hat\Omega)\le \lambda(\Omega_\varepsilon)\le \alpha \frac{\abs{\partial\Omega_\varepsilon}}{\abs{\Omega_\varepsilon}}\to -\infty.
\end{equation*}
\end{remark}

\section{Generalized Minimizing Movements for Laplacian eigenvalues}
\label{sec:RobinEigenvalueAndGMM}
The aim of this section is to apply Theorem \ref{thm:ExistenceGMM} to both eigenvalues on the class of domains which satisfy the so-called uniform $\varepsilon$-cone property. We recall D\'efinition 2.4.1 in \cite{HePi2005}.
Let $y,\xi\in\mathbb R^n$ with $\abs{\xi}=1$ and $\varepsilon>0$. By $C(y,\xi,\varepsilon)$ we denote the cone with vertex at $y$, defined by
\begin{equation*}
C(y,\xi,\varepsilon):=\{ z\in\mathbb R^n\,|\,\langle z-y,\xi\rangle\ge \operatorname{cos}(\varepsilon)\abs{z-y} \text{ and } 0<\abs{z-y}<\varepsilon\}.
\end{equation*}
We say that an open set $\Omega\in\mathbb R^n$ satisfies the \textit{$\varepsilon$-cone condition} if for every $x\in \partial\Omega$ there exists a unit vector $\xi_x$ such that for every $y\in \overline\Omega\cap B_\varepsilon(x)$ the cone $C(y,\xi_x,\varepsilon)$ is contained in $\Omega$. $B_\varepsilon(x)$ denotes the open ball with radius $\varepsilon$, centered at $x$.\\
We define the class
\begin{equation*}
\mathcal O_\varepsilon:=\{\Omega\subset D\,|\,\Omega\text{ is open and satisfies the $\varepsilon$-cone condition}\}
\end{equation*}
for an open ball $D\subset\mathbb R^n$ and a constant $\varepsilon>0$ such that the class is non-empty. The domains in this class have Lipschitz boundary and finite perimeter. On the other hand every open bounded Lipschitz domain with finite perimeter satisfies the $\varepsilon$-cone property for some $\varepsilon>0$. We will also consider domains of given measure and introduce
\begin{equation*}
\mathcal O_{\varepsilon,M}:=\{\Omega\subset D\,|\,\Omega \text{ is open, satisfies the $\varepsilon$-cone condition and } \abs{\Omega}=M \}
\end{equation*}
for an open ball $D\subset\mathbb R^n$ and constants $\varepsilon, M>0$ such that the class is non-empty. \\
We recall the definition of different metrics on these classes. The \textit{Hausdorff metric for compact sets} is given by
\begin{align*}
d^H( A, B)&:= \operatorname{inf}\{\varepsilon \ge 0\,|\,\overline A\subset \overline B_\varepsilon \text{ and }\overline B\subset \overline A_\varepsilon\},\quad \text{where}\\
A_\varepsilon &:=\bigcup_{a\in A}\{z\in \mathbb R^n\,;\,\abs{z-a}\le\varepsilon\},
\end{align*}
the \textit{Hausdorff metric for open sets} is given by
\begin{align*}
d_H( A, B):= d^H(D\backslash \overline A, D\backslash \overline B)
\end{align*}
and a third metric is given by
\begin{equation*}
d_{char}(A,B):=\norm{\chi_{A}-\chi_B}_{L^1(D)},
\end{equation*}
where $\chi_A$ denotes the characteristic function of a measurable set $A$. The classes $\mathcal O_\varepsilon$ and $\mathcal O_{\varepsilon, M}$ together with one of these metrics are complete metric spaces.\\
We state the main result of this section.
\begin{theorem}
The first Laplacian eigenvalue $\lambda\colon \mathcal O_{\varepsilon}\to\mathbb R,\,\Omega\mapsto\lambda(\Omega)$ with Dirichlet or Robin boundary conditions together with one of the metrics $d^H, d_H$ or $d_{char}$ satisfy the assumptions in Theorem \ref{thm:ExistenceGMM}.
\label{thm:ExistenceRobinGMM}
\end{theorem}
\begin{proof}
First of all, the class $\mathcal O_\varepsilon$ is complete with each of the metrics and the eigenvalue is continuous with respect to these metric spaces. We have $\Phi(h_*,x_*;y)\ge 0$ for every $h_*>0,x_*,y_*\in \mathcal O_\varepsilon$.
The compactness of the sub-level sets $\{\Omega\subset \mathcal O_\varepsilon\,|\,\lambda_R(\Omega)\le C\}$ in $(\mathcal O_\varepsilon,d_{char})$ can be proven by similar arguments as in Section 2 of \cite{Deipenbrock2015}. The same holds true for $\lambda_D$. The compactness with respect to the other metrics follows from Th\'eor\`em 2.4.10 in \cite{HePi2005}.
\end{proof}
\begin{corollary}
The first Laplacian eigenvalue $\lambda\colon \mathcal O_{\varepsilon,M}\to\mathbb R,\,\Omega\mapsto\lambda(\Omega)$ with Dirichlet or Robin boundary conditions together with one of the metrics $d^H, d_H$ or $d_{char}$ satisfy the assumptions in Theorem \ref{thm:ExistenceGMM}.
\end{corollary}
\begin{proof}
The proof works similar to the proof of Theorem \ref{thm:ExistenceRobinGMM}. We use the fact, that the Lebesque measure $\mathcal O_{\varepsilon}\to\mathbb R^+,\,\Omega\to \abs{\Omega}$ is continuous.
\end{proof}

\section{Convex bodies and the Brunn-Minkowski inequality}
\label{sec:ConvexbodiesandtheBrunn-Minkowskiinequality}
In this subsection we will use the Brunn-Minkowski inequality for the first Dirichlet eigenvalue to prove $\alpha$-convexity (with $\alpha=0$). 
First, we recall the definition of the class of \textit{convex bodies}
\begin{equation}
\mathcal K:=\{K\subset\mathbb R^n|K \text{ is a non-empty, compact, convex set}\}.
\label{eq:}
\end{equation}
For more details about convex bodies see \cite{Schneider1993}. 
To define a metric on this set, we recall that the \textit{support function} of a convex body is given by
\begin{equation}
\rho_K(x):=\sup\{\langle y,x\rangle|y\in K\}\quad \text{for }x\in\mathbb R^n.
\label{eq:}
\end{equation}
Notice, that the support function is sub-linear and that every sub-linear function from $\mathbb R^n$ to $\mathbb R$ is the support function of a unique convex body. Furthermore, every sub-linear function is uniquely determined by its restriction onto $\mathbb S^{n-1}$. We use these properties to define a metric on $\mathcal K$ by
\begin{equation}
d(K_1,K_2):=\norm{\rho_{K_1}-\rho_{K_2}}_{\mathcal R},
\label{eq:}
\end{equation}
where $\norm{\cdot}_{\mathcal R}$ is a norm on $L^\infty(\mathbb S^{n-1},\mathbb R)$.
The metric defined by $\norm{\cdot}_{\mathcal R}=\norm{\cdot}_{L^p(\mathbb S^{n-1},\mathbb R)},p\in[1,+\infty]$ is called the \textit{$L^p$-metric for convex bodies} and is denoted by $d_{L^p}$. All $L^p$-metrics for convex bodies induce the same topology (see \cite{Vitale1985}) and the $L^\infty$-metric is given by the Hausdorff metric for compact sets. In the following we will use the $L^2$-metric because this allows us to work in the Hilbert space $L^2(\mathbb S^{n-1},\mathbb R)$ and simplifies the $\alpha$-convexity condition in Theorem \ref{thm:ExistenceSemigroups}.\\
To be precise we define the first Dirichlet eigenvalue on a convex body $K\in\mathcal K$ by
\begin{equation*}
\lambda_D(K):=
\begin{cases}
\lambda_D(\operatorname{int}K) &\quad \text{if } \operatorname{int}K\neq\emptyset\\
+\infty&\quad\text{else}.
\end{cases}
\label{eq:}
\end{equation*}
The Brunn-Minkowski inequality for the Dirichlet eigenvalue is given by
\begin{theorem}\label{thm:BMI}
Let $K_0,K_1\in\mathcal K$ and $t\in[0,1]$. Then,
\begin{equation*}
\lambda_D((1-t)K_0\oplus tK_1)^{-1/2}\ge (1-t)\lambda_D(K_0)^{-1/2}+t\lambda_D(K_1)^{-1/2},
\end{equation*}
where $(1-t)K_0\oplus t K_1:=\{(1-t)y_0 + t y_1|y_0\in K_0,y_1\in K_1\}$ denotes the Minkowski sum.
\end{theorem}
Since $z\mapsto z^{-2}$ is concave, this theorem implies
\begin{equation}
\lambda_D((1-t)K_0\oplus tK_1)\le (1-t)\lambda_D(K_0)+t\lambda_D(K_1).\label{thm:weakBMIDirichlet}
\end{equation}
Next, we prove the existence of a contraction semi-group by applying Theorem \ref{thm:ExistenceSemigroups}.
\begin{theorem}
The Dirichlet eigenvalue $\lambda_D\colon \mathcal K\to \mathbb R;\; K\mapsto\lambda_D(K)$ and the metric space $(\mathcal K,d_{L^2})$ satisfy all assumptions of Theorem \ref{thm:ExistenceSemigroups} with $\alpha=0$.
\end{theorem}
\begin{proof}
$(\mathcal K,d_{L^2})$ is a complete metric space (see \cite{Vitale1985}). Obviously, the Dirichlet eigenvalue is proper and Assumption \ref{ass:BoundedFromBelow} is satisfied. We recall the well-known fact, that the domain dependent Dirichlet eigenvalue is continuous with respect to $\gamma$-convergence on convex domains. Since the topology induced by $d_{L^2}$ is stronger than the one induced by $\gamma$-convergence, we get the (lower semi-)continuity of $\lambda_D$. Assumption \ref{ass:AlphaConvexity} follows from Theorem (\ref{thm:weakBMIDirichlet}), Remark \ref{rmk:alphaconvexityHilbert} and the fact that $\mathcal K$ is a convex subset of the Hilbert space $L^2(\mathbb S^{n-1},\mathbb R)$.
\end{proof}
The Brunn-Minkowski inequality for a general domain functional $F$ is given by
\begin{equation*}
F((1-t)K_0\oplus tK_1)^{1/\sigma}\ge (1-t)F(K_0)^{1/\sigma}+tF(K_1)^{1/\sigma}
\end{equation*}
for all $K_0,K_1\in\mathcal K,t\in[0,1]$, where $F\colon \mathcal K\rightarrow[0,+\infty]$ denotes a domain functional which is homogeneous of degree $\sigma\neq 0$.
If the homogeneity is of degree $\sigma\le -1$, we get
\begin{equation}
F((1-t)K_0\oplus tK_1)\ge (1-t)F(K_0)+tF(K_1).
\end{equation}
It turns out that the inequalities are equivalent if the functional $F$ is homogenous with parameter $\sigma\le -1$.\\
Since the first Robin eigenvalue is not homogeneous, we can not expect to have the Brunn-Minkowski inequality for the Robin eigenvalue with $\sigma=-2$. Nevertheless, we formulate the following open problem.
\begin{openproblem}
Let $\lambda_R$ be the first Robin eigenvalue with $\beta>0$.
Does
\begin{equation*}
\lambda_R((1-t)K_0\oplus tK_1)\le (1-t)\lambda_R(K_0)+t\lambda_R(K_1),
\end{equation*}
hold for all $K_0,K_1\in\mathcal K$ and $t\in[0,1]$?
\end{openproblem}
First notice, that the prove of Theorem \ref{thm:BMI} is divided into two steps. The first step proves (\ref{thm:weakBMIDirichlet}). The second step uses a scaling argument to prove the Brunn-Minkwoski inequality. A key argument in the first part is the log-concavity of the eigenfunctions. Since \cite{ACH2017}, we know that the first Robin eigenfunction is not necessarily log-concave. Therefore, other techniques are required to prove or disprove the open problem.

\section{$\alpha$-convexity of the Robin eigenvalue}
\label{sec:AlphaConvexityOfTheRobinEigenvalue}
In this subsection, we will use the second domain variation of the first Robin eigenvalue to prove $\alpha$-convexity. More precisely, we give a uniform estimate for the second domain variation. 
To define the second domain variation, we consider families of perturbations $(\Omega_t)_t$ for a fixed bounded, open set $\Omega\subset \mathbb R^n$. For small $t>0$, we set 
\begin{equation}
\Omega_t:=\{y=x+tv(x)+\tfrac{t^2}{2}w(x)+o(t^2)|x\in\Omega\}
\label{eq:}
\end{equation}
where $v,w\colon \Omega\mapsto\mathbb R^n$ are sufficient smooth vector fields.
In the next step, we define
\begin{equation*}
\lambda_R[\Omega,v,w](t):=\lambda_R(\Omega_t)
\end{equation*}
and compute the first two derivatives at $t=0$
\begin{equation*}
\dot\lambda_R[\Omega,v,w](0):=\frac{\operatorname{d}}{\operatorname{d}\!t}\lambda_R(\Omega_t)\bigg\vert_{t=0} \quad \text{and}\quad \ddot\lambda_R[\Omega,v,w](0):=\frac{\operatorname{d}^2}{\operatorname{d}\! t^2}\lambda_R(\Omega_t)\bigg\vert_{t=0}.
\end{equation*}
They are called the \textit{first and second domain variation}. Bandle and Wagner \cite{Bandle2014} proved the following formula for the first variation
\begin{equation}
\dot\lambda_R[\Omega,v,w](0)=\bint{\left(|\nabla u|^2-\lambda(\Omega)u^2-2\beta^2u^2+\beta(n-1)H u^2\right)(v\cdot\nu)},
\label{eq:}
\end{equation}
where $u$ denotes the first normalized Robin eigenfunction on $\Omega$ and $H$ denotes the mean curvature of $\partial\Omega$. 
Observe that the computation of the second domain variation in \cite{Bandle2014} assumes that $\Omega$ is a critical point ($\dot\lambda_R[\Omega,v,w]=0$). Slight adjustments on these calculations give a similar result for general domains with an additional term which depends on $\dot\lambda_R[\Omega,v,w]$. The computations and the formula are given in \cite{Holle2018}. These computations require that the boundary of $\Omega$ is of class $C^{2,\gamma}$ and that the vector fields $v,w$ are in $C^{2,\gamma}(\Omega,\mathbb R)$ for a H\"older parameter $\gamma\in(0,1]$.\\
To obtain a uniform estimate on $\ddot\lambda_R[\Omega,v,w]$, we restrict ourselves to domains given by
\begin{equation}
\Omega(\eta):=\{r\theta|\theta\in\mathbb S^{n-1},0\le r<\eta(\theta)\}
\label{eq:}
\end{equation}
where $\eta\colon\mathbb S^{n-1}\rightarrow [R_*,+\infty)$ are suitable boundary parametrization.
\begin{definition}
Let $0<R_*\le R^*, \gamma\in(0,1]$ and $M\in (0,+\infty]$ such that $\abs{B_{R_*}}\le M$. We set
\begin{align*}
\mathcal O:=\{\Omega\subset\mathbb R^n\,|\,\Omega=\Omega(\eta) \text{, } \eta \text{ satisfies (O1)--(O2)}\}\\
d_{\mathcal O}(\Omega(\eta_1),\Omega(\eta_2)):=\left\|\eta_1-\eta_2\right\|_{W^{1,2}(\mathbb S^{n-1})}
\end{align*}
with
\begin{align}
\tag{O1} \eta\in C^{2,\gamma}(\mathbb S^{n-1},[R_*,+\infty))\quad &\text{and}\quad \norm{\eta}_{C^{2,\gamma}(\mathbb S^{n-1})}\le R^*,\\
\tag{O2} \abs{\Omega(\eta)}&\le M.
\end{align}
\end{definition}

We get the following estimate for the second domain variation.
\begin{theorem}\label{thm:upperboundeigenvalue}
The second domain variation of the first Robin eigenvalue is bounded by
\begin{equation}\label{eq:EstimateSecondVariationinTheorem}
\lvert\ddot\lambda_R[\Omega,v,w](0)\rvert\le C (\norm{v}^2_{W^{1,2}(\partial\Omega)}+\norm{w}_{L^1(\partial\Omega)}),
\end{equation}
where $C>0$ is independent of the domain $\Omega\in\mathcal O$ and the vector fields $v, w\in C^{2,\gamma}(\Omega,\mathbb R^n)$.
\end{theorem}

\begin{proof}
Since \cite[Theorem 3.1]{Holle2018}, we have
\begin{align}
\ddot \lambda_R[\Omega,v,w](0)=P(v)+Q(u')+\dot \lambda[\Omega,w,0](0),
\end{align}
where $P$ is a quadratic form in $v$ and $Q$ is a quadratic form in $u'$.\\
\underline{Step 1:} Before we start with the estimate, notice that $|\Omega|,\,|\partial\Omega|,\, H$ and $\lambda(\Omega)$ are uniformly bounded for $\Omega\in\mathcal O$. Furthermore, the constant in the trace theorem ($W^{1,2}(\Omega)\to L^2(\partial\Omega)$) is uniformly bounded since \cite{BrPa}. \\
We need uniform regularity bounds on $u$. Since Theorem 4.2 in \cite{Agmon1962}, we have $u\in W^{2,p}(\Omega)$ for all $p\in(1,+\infty)$. It can be shown that this regularity holds uniformly on $\mathcal O$. This implies uniform $C^{1,\gamma}$-H{\"o}lder-regularity in the interior and uniform $C^{\gamma}$-H{\"o}lder-regularity on the boundary, since $\partial \Omega$ is sufficiently smooth. A classical Schauder argument (e.g. Theorem 2.19 in \cite{GGS1991}) implies that
\begin{equation*}
\|u\|_{C^{2,\gamma}(\bar\Omega)}\le C,\quad\text{for all } \Omega\in\mathcal O.
\end{equation*}
\underline{Step 2:} We continue with an estimate for $\dot\lambda$. Applying H\"older's inequality and the uniform bounds give
\begin{equation*}
|\dot\lambda[\Omega,v,w]|\le C \norm{v}_{L^1(\partial\Omega)}.
\end{equation*}
With similar straight-forward but lengthy computations, we get
\begin{equation}
|P(v)|\le C \norm{v}^2_{W^{1,2}(\partial\Omega)}.
\end{equation}
Notice that a by-product of these computation is a uniform estimate on the second domain variation of the perimeter.\\
\underline{Step 3:} The most challenging part is the estimate for the quadratic form
\begin{equation*}
Q(u')=\vint{|\nabla u'|^2}-\lambda_R[\Omega,v,w](0)\vint{u'^2}+\beta\bint{u'^2},
\end{equation*}
where $u'$ is the shape derivative. We consider the problem
\begin{align}\label{equphi1}
\begin{split}
\Delta \phi+\lambda_R[\Omega,v,w]\phi&=-\dot\lambda_R[\Omega,v,w](0)u \quad\text{in }\Omega,\\
\partial_{\nu}\phi+\beta \phi&=b(v):=\partial_\nu(v\cdot\nabla u)+\nabla^\tau u\cdot D_v\nu+\nu\cdot D_v\nabla u-\beta(v\cdot \nabla u)\quad\text{on }\partial\Omega,\\
\vint{u\phi}&=0,
\end{split}
\end{align}
with the unique solution $\phi:=u'-\int_\Omega u' u\;\text{d}x \,u$. We want to prove 
\begin{equation}\label{eq:ProveBoundedW12NorminProofTheoremalphaconvexity}
\norm{\phi}_{W^{1,2}(\Omega)}\le C \norm{b(v)}_{L^2(\partial\Omega)}\le C \lVert v\rVert^2_{W^{1,2}(\Omega)} \quad\text{for all } \Omega\in \mathcal O,
\end{equation}
since we get
\begin{equation*}
Q(u')=\bint{b(v)\phi}
\end{equation*}
by integration by parts and (\ref{eq:EstimateSecondVariationinTheorem}) follows.\\
The prove works similar to the proof of classical a priori estimates for elliptic equations. First, we prove 
\begin{equation}\label{eq:Estimatephil2normwithcontradiction}
\norm{\phi}_{L^2(\Omega)}\le C\norm{b(v)}_{L^2(\partial\Omega)} \quad\text{for all } \Omega\in \mathcal O,
\end{equation}
and assume that such a constant does not exist. Then, there exist sequences $(\Omega_k)_k\subset \mathcal O$ and $(v_k)_k$, $v_k\in C^{2,\gamma}(\Omega_k, \mathbb R^n)$ with
\begin{equation}
\norm{\phi_k}_{L^2(\Omega_k)}\ge k\norm{b_k(v_k)}_{L^2(\partial\Omega_k)}.
\end{equation}
We rescale the solution $\phi_k$ and the linear functional $b_k$ by
\begin{equation*}
\hat\phi_k:=\frac{\phi_k}{\norm{\phi_k}_{L^2(\Omega_k)}} \quad \text{and}\quad \hat b_k:=\frac{b_k}{\norm{\phi_k}_{L^2(\Omega_k)}}.
\end{equation*}
Next, we show that the $L^2$-norm of $\nabla\hat\phi_k$ is bounded. Testing (\ref{equphi1}) by $\hat\phi_k$, integration by parts and $\beta>0$ give
\begin{align}
\vintk{|\nabla \hat\phi_k|^2}&=-\beta \bintk{\hat\phi_k^2}+\lambda_R(\Omega_k) \vintk{\hat\phi_k^2}+\bintk{\hat b_k(v_k)\hat \phi_k}\nonumber\\
&\le C+\lVert\hat b_k(v_k) \rVert_{L^2(\partial\Omega_k) } \lVert\hat \phi_k \rVert_{L^2(\partial\Omega_k)}\nonumber\\
&\le C\left(1+\left(\vintk{|\nabla \hat\phi_k|^2}\right)^{1/2}\right).\label{equgradphiharbounded}
\end{align}
This uniform inequality implies that $\lVert \nabla \hat\phi_k \rVert_{L^2(\Omega_k)}$ is bounded by a constant independent of $k$.
We summarize that $\lVert\hat \phi_k\rVert_{W^{1,2}(\Omega_k)}$ is uniformly bounded.\\
To lead the assumption to a contradiction, we consider the gap between the first and second Robin eigenvalue $\lambda_{R,1}$ and $\lambda_{R,2}$. In the following $u_k$ denotes the normalized, first Robin eigenfunction on $\Omega_k$. With the min-max principle and 
\begin{align*}
\vintk{\hat\phi_k\, u_k}=&0, \quad \text{we have} \quad \lambda_{R,2}(\Omega_k)-\lambda_{R,1}(\Omega_k)\le\mathcal R(\hat\phi_k,\beta,\Omega_k)-\lambda_{R,1}(\Omega_k),\\
\text{where}\quad &\mathcal R(\hat\phi_k,\beta,\Omega_k)=\frac{\int_{\Omega_k}|\nabla v|^2\;\text{d}x+\beta\oint_{\partial\Omega_k}v^2\;\text{d}S}{\int_{\Omega_k}v^2\;\text{d}x}.
\end{align*}
This, the definition of $\hat\phi_k$ and the uniform estimates give
\begin{align*}
\lambda_{R,2}(\Omega_k)-\lambda_{R,1}(\Omega_k)&\le\left(\mathcal R(\hat\phi_k,\beta,\Omega_k)-\lambda_{R,1}(\Omega_k)\right)\vintk{\hat\phi_k^2}\\
&=\vintk{|\nabla\hat\phi_k|^2}+\beta\bintk{\hat\phi_k^2}-\lambda_{R,1}(\Omega_k)\vintk{\hat\phi_k^2}\\
&=\bintk{(\partial_\nu \hat\phi_k+\beta\hat\phi_k) \hat\phi_k}+\dot \lambda_{R,1}[\Omega_k,v,w](0)\frac{1}{\norm{\phi_k}_{L^2(\Omega_k)}}\vintk{\hat\phi_k u_k}\\
&=\bintk{\hat b_k(v_k)\hat\phi_k}\\
&\le \frac{C}{k}\rightarrow 0 \qquad\text{as} \;k\rightarrow \infty.
\end{align*}
Since the continuity of the eigenvalues and the compactness of $\mathcal O$ w.r.t. the Hausdorff metric, there exists $\Omega^*\in\mathcal O$ such that
\begin{equation*}
\lambda_{R,2}(\Omega^*)-\lambda_{R,1}(\Omega^*)=0.
\end{equation*}
This is a contradiction because the first Robin eigenvalue of a connected domain is simple. So, we proved the uniform a priori estimate (\ref{eq:Estimatephil2normwithcontradiction}). Finally, we prove (\ref{eq:ProveBoundedW12NorminProofTheoremalphaconvexity}) with similar arguments as in (\ref{equgradphiharbounded}).
\end{proof}

\begin{theorem}
Let $\lambda$ be the first Robin eigenvalue. Then, there exists an $\alpha\in \mathbb R$ such that
\begin{align*}
\lambda_R(\Omega_t)\le(1-t)\lambda_R(\Omega_0)+t\lambda_R(\Omega_1)-\frac{\alpha}{2}t(1-t)d^2_{\mathcal O}(\Omega_0,\Omega_1)\quad \text{for all }\Omega_0,\Omega_1\in\mathcal O,
\end{align*}
where $\Omega_0=\Omega(\eta_0), \Omega_1=\Omega(\eta_1)$ and $\Omega_t=\Omega((1-t)\eta_0+t\eta_1)$.
\label{thm:AlphaConvexityNearlySpherical}
\end{theorem}
\begin{proof}
We consider
\begin{equation*}
h\colon [0,1]\rightarrow \mathbb R,\quad t\mapsto \lambda_R(\Omega_t) \quad \text{ for } \Omega_1,\Omega_2\in\mathcal O.
\end{equation*}
The second derivative of $h$ is given by
\begin{align*}
\ddot h(t)=\ddot \lambda_R[\Omega_t,\tilde{v}_0,0](0),
\end{align*}
where $\tilde{v}_0$ is a $C^{2,\gamma}(\Omega_t,\mathbb R^n)$ vector field with $\tilde{v}_0(x)=v_0(x):=\tfrac{\eta_1-\eta_0}{\eta_t}\!\left(\tfrac{x}{\abs{x}}\right)x$ in a neighborhood of $\partial\Omega$. Theorem \ref{thm:upperboundeigenvalue} implies
\begin{equation*}
|\ddot h(t)|=|\ddot \lambda_R[\Omega(\eta_t),\tilde{v}_0,0](0)|\le C_{\ddot\lambda}\norm{\tilde{v}_0}_{W^{1,2}(\partial\Omega_t)}=C_{\ddot\lambda}\norm{v_0}_{W^{1,2}(\partial\Omega_t)}.
\end{equation*}
A straight-forward computation gives
\begin{equation*}
\norm{v_0}_{W^{1,2}(\partial\Omega_t)}\le C\norm{\eta_0-\eta_1}_{W^{1,2}(\mathbb S^{n-1})} \quad \text{for all }\Omega_0,\Omega_1\in\mathcal O, t\in[0,1].
\end{equation*}
We obtain
\begin{equation*}
\ddot h(t)\ge -|\ddot h(t)|\ge \alpha \norm{\eta_1-\eta_0}_{W^{1,2}(\mathbb S^{n-1})},
\end{equation*}
where $\alpha$ is independent of $\Omega_0,\Omega_1\in\mathcal O$ and $t\in[0,1]$.
This implies the $\alpha$-convexity inequality with $\alpha<0$.
\end{proof}
We use this convexity property to prove existence of a contractive semi-group.
\begin{theorem}
The first Robin eigenvalue $\lambda_R\colon \mathcal O\to \mathbb R;\;\Omega\mapsto\lambda_R(\Omega)$ and the metric space $(\mathcal O,d_{\mathcal O})$ satisfies all assumptions in Theorem \ref{thm:ExistenceSemigroups}.
\label{thm:ExistenceNearlySpherical}
\end{theorem}
\begin{proof}
$(\mathcal O,d_{\mathcal O})$ is isometrically isomorphic to a closed subset $\mathcal P$ of the Hilbert space $W^{1,2}(\mathbb S^{n-1})$ by $\mathcal O\to\mathcal P,\,\Omega(\eta)\mapsto\eta$. This implies that $(\mathcal O,d_{\mathcal O})$ is a complete metric space. Obviously, the Robin eigenvalue is proper and Assumption \ref{ass:BoundedFromBelow} is satisfied. We recall the fact, that the Robin eigenvalue is continuous with respect to $\gamma$-convergence on domains which satisfy the uniform $\epsilon$-cone property. Since the topology induced by $d_{\mathcal O}$ is stronger than the one induced by $\gamma$-convergence and since $\mathcal O$ satisfies the uniform $\epsilon$-cone property, we get the lower semi-continuity of $\lambda$. Assumption \ref{ass:AlphaConvexity} follows from Theorem \ref{thm:AlphaConvexityNearlySpherical}, Remark \ref{rmk:alphaconvexityHilbert} and the fact that $\mathcal O$ is isometrically isomorphic to a convex subset of a Hilbert space.
\end{proof}
\begin{remark}
We explain, why we take the $W^{1,2}$-norm in the definition of the metric. The uniform estimate in Theorem \ref{thm:upperboundeigenvalue} is also valid for $-\lambda_R$ and we can use the arguments above to construct a contraction semi-group for $-\lambda_R$. A similar argument as in Remark \ref{rmk:negativebeta} shows that it is necessary to measure the boundary oscillations. On the other hand the estimate in Theorem \ref{thm:upperboundeigenvalue} is more than we need and a lower bound on $\ddot\lambda_R$ is enough.
\end{remark}

\bibliographystyle{plain}
\bibliography{Literatur}
\end{document}